\documentclass{amsart}

\usepackage{enumerate}
\usepackage{amssymb}
\usepackage{amsmath,amscd}
\usepackage{amsthm}
\usepackage[colorinlistoftodos, disable]{todonotes}

\newcommand{\diam}{\operatorname{diam}}
\newcommand{\vol}{\operatorname{vol}}
\newcommand{\Isom}{\operatorname{Isom}}

\newcommand{\Lg}{\mbox{$\mathfrak g$}}
\newcommand{\Lt}{\mbox{$\mathfrak t$}}

\newcommand{\Out}{\operatorname{Out}}
\newcommand{\Aut}{\operatorname{Aut}}
\newcommand{\Hom}{\operatorname{Hom}}

\newcommand{\Z}{\mathbb{Z}}

\newcommand{\R}{\mathbb{R}}
\newcommand{\C}{\mathbb{C}}
\newcommand{\HH}{\mathbb{H}}
\newcommand{\FF}{\mathbb{F}}

\newcommand{\OO}{\operatorname{O}}
\newcommand{\SO}{\operatorname{SO}}
\newcommand{\PSO}{\operatorname{PSO}}
\newcommand{\U}{\operatorname{U}}
\newcommand{\SU}{\operatorname{SU}}
\newcommand{\PSU}{\operatorname{PSU}}
\newcommand{\Sp}{\operatorname{Sp}}
\newcommand{\PSp}{\operatorname{PSp}}
\newcommand{\PSL}{\operatorname{PSL}}
\newcommand{\PGL}{\operatorname{PGL}}
\newcommand{\GL}{\operatorname{GL}}
\newcommand{\Ss}{\mathbb{S}}

\newcommand{\Co}{c}

\newtheorem{theorem}{Theorem}

\newtheorem{lemma}[theorem]{Lemma}
\newtheorem{proposition}[theorem]{Proposition}
\newtheorem{quest}[theorem]{Question}

\newtheorem*{maintheorem}{Main Theorem}

\theoremstyle{definition}
\newtheorem{definition}[theorem]{Definition}

\theoremstyle{remark}
\newtheorem{remark}[theorem]{Remark}
\newtheorem{example}[theorem]{Example}

\title[A diameter gap for quotients of the unit sphere]{A diameter gap for quotients of the unit sphere}
\date{}

\author[C.~Gorodski]{Claudio Gorodski}
\address{University of S\~ao Paulo, Brazil}
\email{gorodski@ime.usp.br}

\author[C.~Lange]{Christian Lange}
\address{University of Munich, Germany}
\email{lange@math.lmu.de}

\author[A.~Lytchak]{Alexander Lytchak}
\address{University of Cologne, Germany}
\email{alytchak@math.uni-koeln.de}

\author[R.~Mendes]{Ricardo A. E. Mendes}
\address{University of Oklahoma, USA}
\email{Ricardo.Mendes@ou.edu}

\begin{document}


\begin{abstract}
We prove that for any isometric action of a group on  a  unit sphere of dimension larger than one, the quotient space has diameter zero or larger than a universal dimension-independent positive constant.
\end{abstract}

\maketitle


\section{Introduction}

\subsection{Main result} We prove the following gap theorem, answering a question  going  back to Karsten Grove  and investigated in \cite{McGowan93} and \cite{Greenwald00}.

\begin{maintheorem}\label{MT}
There exists some $\epsilon >0$ such that for any $n\geq 2$ and any group $G$
acting by isometries on the unit $n$-dimensional sphere $\Ss^n$  the quotient space $\Ss^n/G$ either has diameter $0$ or at least $\epsilon$. 		
\end{maintheorem}

{Note that the diameter of the quotient space does not change if the group $G$ is replaced by its closure $\bar G$ and that for a closed group $G=\bar G$ the quotient space $S^n/G$ is an Alexandrov space of curvature bounded below by one.}
The diameter of $\Ss^n/G$ is $0$ if and only if any orbit of $G$ is 
dense in $\Ss^n$, thus  if the closure  $\bar G$  of $G$ in $\OO(n)$ acts transitively on $\Ss^n$.


The case $n=1$ needs to be excluded,  since quotients of $\Ss^1$ by the action of  cyclic groups can have  arbitrary small diameter.

The existence of a dimension-dependent bound   $\epsilon(n)$ has been proved in \cite{Greenwald00}. There, earlier in \cite{McGowan93} and later in  \cite{McGowanSearle05} and \cite{DGMS09} explicit lower bounds on the diameter
have been found for some special classes of actions. Most notably, a lower 
bound $\alpha >\frac 1 4$  has been verified by explicit calculation
for free actions, for finite Coxeter groups and for actions with quotients of dimension $\leq 2$. Recently in an independent preprint   the existence of such a lower bound was  proved \todo{proved instead of claimed}  for \emph{unitary} actions of  \emph{connected} Lie groups with the exception of spin representations \todo{spin representations instead of spin groups} \cite{GauthierRossi18}\todo{added non-spin assumption, removed sentence about a gap in the proof}.  We refer the reader to \cite{GauthierRossi18} for the relevance of this problem to Control Theory.  After finishing this paper we have learned from Ben Green about his recent work \cite{BGreen}, in which a very  strong version of our main theorem was verified for 
\emph{finite} groups $G$.  In particular,  he proved   for such groups that  $\epsilon (n) $ converges to $\pi/2$   if $n$ converges to infinity.

Unlike the existence of a  dimension-dependent constant $\epsilon (n)$ from \cite{Greenwald00}, our dimension-independent bound cannot be derived by a limiting argument. A related fact is that no such lower bound exists for isometric actions on the unit sphere in an infinite-dimensional Hilbert space~\cite{We}. While we have not tried to determine our constant explicitly, the proof indeed provides some explicit bound on $\epsilon$ in the Main Theorem\todo{Replaced Theorem A with Main Theorem}.  In a future work, we hope to bring this explicit bound in a range comparable with the existing examples,  see \cite{DGMS09} for some conjectures about the optimal value of $\epsilon$,  resolved and improved for finite groups in \cite{BGreen}.

\subsection{Related questions}

We start with a generalization of Greenwald's dimension-dependent bound \cite{Greenwald00}  and characterize compact Riemannian manifolds for which such a bound exists. The proof relies on a limiting argument similar to the one employed by Greenwald.
\begin{theorem}\label{MT:finitepi1}
Let $M$ be a compact homogeneous Riemannian manifold. Then $\pi_1(M)$ is finite if and only if there exists $\epsilon_M>0$ such that, for every subgroup $G\subset\Isom(M)$, either $\diam(M/G)=0$ or $\diam(M/G)>\epsilon_M$.
\end{theorem}

Unlike the Main Theorem, the bound in Theorem \ref{MT:finitepi1} above cannot be made independent of the space $M$, even after the metric is rescaled to have a fixed diameter. A counter-example is given by the groups $\SO(n)$, see Example \ref{E:SO(n)}. Nevertheless, the following problem is likely to have an affirmative answer:

\begin{quest}
Does there exist a lower bound for the diameter of quotients of simply-connected compact symmetric spaces depending only on the rank? 
\end{quest}

Note that the rank one case follows from the Main Theorem and Theorem \ref{MT:finitepi1}.

Every Riemannian orbifold with constant curvature one is a good orbifold and
therefore a quotient of the unit sphere. Therefore, a special case of the Main Theorem is the existence of a universal lower bound on the diameter of Riemannian orbifolds with constant curvature one (in the case of manifolds this is the main result in \cite{McGowan93}). Considering curvature negative one instead of one yields the following natural question, which seems to be open: 
\begin{quest}
  Is there a universal lower bound for the diameter of hyperbolic manifolds
  (resp.~orbifolds)?
\end{quest}
Note that the existence of a \emph{dimension-dependent} bound follows from
Margulis' Lemma, see e.g. \cite[Cor.~1, \S12.7]{Ra}.

While our proof of the Main Theorem uses several geometric arguments,
it heavily relies on the structure and classification of compact Lie groups and their representations. Even in the connected case it 
seems to be a difficult task to remove the representation-theoretic arguments from the proof and obtain an affirmative answer to the following:
 
\begin{quest} \label{question}
 	Does there exist a universal constant $\epsilon$ such that for any non-trivial singular Riemannian foliation $\mathcal F$  on a unit sphere 
 	$\Ss^n$, the quotient $\Ss^n/\mathcal F$ has diameter at least $\epsilon$?
\end{quest} 

We refer to \cite{Molino}, \cite{RadeschiSRF} for the theory of singular Riemannian foliations, being a group-free generalization of isometric group actions and to \cite{LytchakRadeschi18}, \cite{MendesRadeschi16} for algebraic   properties of singular Riemannian foliations on spheres. While in codimension one a positive answer to the above problem is a famous theorem of M\"unzner \cite{Muenzner81},  nothing is known in higher \emph{codimensions}.  Even the existence of a dimension-dependent bound $\epsilon (n)$ is presently not known. Nevertheless, we note that Question \todo{replaced Problem with Question} \ref{question} has an affirmative answer for all \emph{currently known} examples, because these are all constructed starting from a homogeneous foliation, and repeatedly composing it with Clifford foliations, see \cite{Radeschi14}.

\subsection{The proof of the Main Theorem}
We are going to explain the main steps involved in the proof of the Main Theorem now. 

First, we may replace $G$ by the closure of its image in $\OO(n+1)$, thus we may assume $G$ to be compact.  

If the representation of $G$ on $V=\R^{n+1}$ is reducible then the quotient 
$\Ss^n /G$ has diameter $\frac {\pi} 2$ or $ \pi$, see Lemma \ref{L:reducibility}.
Using a slightly more refined argument, we deduce that the existence of a normal subgroup $N$ of $G$ acting reducibly on $ V$ implies that $N$ acts as (real, complex, or quaternionic) scalars or that the diameter of $\Ss^n /G$ is at least $\frac \pi {4}$, see Lemma \ref{L:super-reducible}.
  
  Replacing $G$ by a larger group can only decrease diameter. Combining this with the previous observation and ruling out 4 special classes of examples by hand, we reduce the task to the following two main cases.
  
  I) The group $G$ is a (uniformly) finite extension of $G_0\times G_1$ where $G_0$ is the group of $F$-scalars (with $F=\R, \C, \HH$) and $G_1$ is  a  simple, simply connected  compact  group acting on $V$ irreducibly.
  
  II) The connected component $G^0$ of $G$ acts as $F$-scalars on $V$, thus,
  $G$ is finite up to scalars.
  
  In case (I), we invoke the algebraic result proved in \cite{GorodskiSaturnino17}, saying that the orbit $G\cdot p$ of $G$ through the \emph{highest weight vector} has a \emph{focal radius} bounded
  from below by a universal constant. Then we combine this with a quotient version of the Klingenberg injectivity radius estimate (Proposition \ref{P:Klingenberg}) to finish the proof.

 The technically  more involved case (II) can be deduced from the technically much more complicated paper \cite{BGreen}.   Following this path one could dispose of Section \ref{S:super-reducible} and of Appendix \ref{A:l(S)}
below. We have decided  to keep  our proof of this result whose simple geometric idea we are going to explain now. The reader can  very well dispose of this idea  and take the shortcut explained in Remark \ref{rmk: new} below.

  We only explain the main idea of the proof of case (II),
  neglecting all difficulties arising from the presence of scalars, which force us to work with projective representations rather than actual representations.  Thus we assume that $G$ is a finite group and that it is a maximal subgroup of $\OO(n)$, in particular, the representation is of real type.
  
   We compare the diameter and the volume of the quotient $\Ss^n/G$.   The volume equals $\vol (\Ss^n)/|G|$. On the other hand, 
   by the theorem of Bishop--Gromov, 
  the volume of $\Ss^n /|G|$ is bounded from above by $\Co\cdot r^n \cdot \vol (\Ss^n)$, where $\Co$ is a universal constant and 
   $r$ is  the diameter of $\Ss^n/G$. Thus, in order to obtain the conclusion, we only need to verify that $\log (|G|) /n$  has a universal upper bound (for all representations that we cannot rule out by other means).
   
   If the group $G$ is a finite  simple group, then the classification of such groups and existing lower bounds on the dimension of their representations provide us with the needed bound (with the only exception of the minimal representation of the alternating group, for which we already have the bound of  $\frac \pi 4$, \cite{Greenwald00}).  If the group $G$ is not simple, we consider a minimal normal subgroup $N$ of $G$, use the fact that this  normal subgroup must act irreducibly and that $G/N$ (again up to scalars) embeds into the group of outer automorphisms of $N$.  Since $N$ must be a power of a simple group, we again apply the classification of finite simple groups and obtain the required bound on  $\log (|G|) /n$.

   \begin{remark} \label{rmk: new}
   We are going to explain how case II follows directly from \cite{BGreen}.
   If $G^0=\{1\}$, then $G$ is finite and the main result of \cite{BGreen} states that 
   the diameter of $\Ss ^n /G$ is at least $ \epsilon _\textrm{finite}  (n)$ with $\lim _{n\to \infty} \epsilon _\textrm{finite} (n) =\pi /2 $.  
   
   If $G^0 = \U (1)$,  then  approximating $\U (1)$ by finite cyclic subgroups, we obtain an approximation of $G$ by finite
   subgroups. Therefore, also in this case we get from \cite{BGreen}  that the diameter of $\Ss ^n /G$ is at least $ \epsilon _\textrm{finite}  (n)$ with $\epsilon _\textrm{finite} (n)$ as above.
   
  Finally, if $G^0$ is  $\Sp (1)$ acting as quaternionic scalars, we can write 
   $G =G^0\cdot \Gamma$, where the finite group $\Gamma$ is the centralizer of $G^0$ in $G$. Then we consider a fixed finite subgroup $H$ of $\Sp (1)$ which acts irreducibly on
  $\Ss^3$, for instance the binary icosaedral group.  The main theorem of \cite{BGreen}
  implies that the diameter of $\Ss^n/ (H \cdot \Gamma)$ is bounded from below
  by $\epsilon _\textrm{finite} (n)$ as above.  However, in any $G^0$-orbit, which is a round sphere $\Ss ^3$,
  the corresponding  $H$-orbit is $s$-dense,  where $s<\pi /2$ is  the diameter of $\Ss ^3/H$.
  Then, we get $\epsilon _\textrm{finite} (n)- s$ as  a lower bound for the diameter of $\Ss^n /G$.
  By \cite{BGreen} this number is bounded away from $0$ for $n$ large enough.
   \end{remark}

   \begin{remark}
   		If one is interested in the case of connected Lie groups only, the proof can be considerably shortened. Indeed, by passing to a maximal non-transitive  connected closed group
   		and using Theorem~\ref{T:reflection} to discard polar actions
                (see section~\ref{S:polar} for this concept) we directly  arrive  at one of first three cases in
                Lemma~\ref{L:G^0}.
                Subsections~\ref{inj-rad} and~\ref{i-ii-iii} cover these cases and we obtain as a lower bound for $\epsilon$  half the focal radius of the orbit through the special point, hence $   \approx  \frac 1 {30}$ according to \cite{GorodskiSaturnino17}.
              \end{remark}

              Understanding the diameter of quotients of unit spheres
              also has a bearing on the global structure and
              classification of compact positively and non-negatively
              curved manifolds (compare~\cite{grove}). Indeed, the
              orbit space of such a manifold
              under the action of a compact Lie group of isometries is an
              Alexandrov space of 
              positive (resp.~non-negative) curvature,
              whose local geometry is controlled by its tangent cones which,
in turn, are determined by the associated isotropy representations. In this
sense, the Main Theorem says that ``Riemannian orbit spaces cannot
be arbitrarily singular''.

\subsection{Organization}
In Section \ref{S:prelim} we recall a couple of basic facts and definitions about real, complex and quaternionic representations, which are used throughout this article, as well as some known facts about diameter of quotients which we will use later.  Section \ref{S:polar} 
concerns normal subgroups and reduces the proof of the Main Theorem to two cases, according to whether the identity component $G^0$ of the given group $G$ acts irreducibly, or as scalar multiplication. Section \ref{S:super-reducible} finishes the proof when $G^0$ acts as scalar multiplication, that is, when $G$ is essentially a finite group, while Section \ref{S:irreducible} deals with the case where $G^0$ acts irreducibly. Section \ref{S:generalizations} is devoted to the  proof of Theorem \ref{MT:finitepi1}. 

Finally, Appendices \ref{A:l(S)} and \ref{A:G^0} contain proofs of two technical but essentially known Lemmas needed in Sections \ref{S:super-reducible} and \ref{S:irreducible}, respectively.

\bigskip

\noindent\textbf{Acknowledgments.} We would like to thank Anton Petrunin for popularizing the question on Mathoverflow and, subsequently, Nik Weaver for providing a counterexample in infinite dimensions~\cite{We}.  We thank Jean-Paul Gauthier, Karsten Grove and an anonymous referee for helpful comments and remarks. We are grateful to Ben Green for pointing out close connections to his recent work. 

The first named author has been partially supported by
the CNPq grant 302882/ 2017-0 and the FAPESP project  	
2016/23746-6, the second named author by the DFG funded project SFB/TRR 191, the third named author by the DFG funded project no. 281071066, TRR 19 and the fourth named author by the DFG grant ME 4801/1-1 and the NSF grant DMS-2005373.

\section{Preliminaries}
\label{S:prelim}

\subsection{Representations of real, complex, and quaternionic types}
In this section we briefly collect a few definitions and basic facts about representations over $\R,\C,\HH$, of real, complex and quaternionic types that are used throughout the present article. A thorough treatment can be found in  \cite[Section 2.6]{BroeckertomDieck}.

Let $G$ be a compact group. A \emph{real representation} of $G$ is a group homomorphism $G\to \GL(U)$, where $U$ is a real vector space. It is called \emph{irreducible} when the only $G$-invariant real subspaces are $\{0\}$ and $U$. In this case, Schur's lemma implies that the algebra of all $G$-equivariant endomorphisms of $U$ is a real associative division algebra, which, by Frobenius' Theorem, must be isomorphic to $\R$, $\C$, or $\HH$. This representation is then called of \emph{real}, \emph{complex}, or \emph{quaternionic type}, respectively.

A \emph{complex representation} of $G$ is a group homomorphism $G\to \GL(V)$, where $V$ is a complex vector space, and it is called { \emph{irreducible} when the only $G$-invariant complex subspaces of $V$ are $\{0\}$, $V$. In this case, it is called of \emph{real type} (resp. \emph{quaternionic type}) when it admits a \emph{real structure} (resp. \emph{quaternionic structure}), that is, a $G$-equivariant conjugate-linear map $\epsilon:V\to V$ with $\epsilon^2=1$ (resp. $\epsilon^2=-1$). The representation $V$ is called of { \emph{complex type} when it admits neither a real nor a quaternionic structure, or equivalently, when $V$ is not isomorphic to the complex-conjugate representation $\bar{V}$.

If $U$ is a real irreducible representation of real type, then its complexification $V=\C\otimes_\R U$ (that is, the $G$-module obtained by extension of scalars) is a complex irreducible representation of real type. Conversely, given a complex irreducible representation of real type $V$, with real structure $\epsilon$, then the fixed-point set $U$ of $\epsilon$ is a real irreducible representation of real type, called the  \emph{real form} of $V$.

On the other hand, if $U$ is a real irreducible representation of complex (resp. quaternionic) type, then it is the \emph{realification} of an irreducible complex representation $V$ of complex (resp. quaternionic) type, that is, it is obtained from $V$ by restriction of scalars (from $\C$ to $\R$).

\subsection{Diameter of quotients}

 Here we collect some basic facts about the diameter of quotients.
%
 We start with a well-known result, whose  proof can be found, for instance,  in \cite[pages 75--76]{GorodskiLytchak14}.
\begin{lemma}
\label{L:reducibility}
Let $G\subset\OO(n)$. Then the diameter of $\Ss^{n-1}/G$ is equal to $\pi$ if and only $G$ fixes some non-zero vector. Otherwise the diameter is less than or equal to $\pi/2$, with equality precisely when the representation of $G$ on $\R^n$ is reducible.
\end{lemma}

Next we turn to the behaviour of the diameter of the quotient with respect to inclusion of groups $K\subset G\subset \OO(n)$. A simple fact we will use frequently is that $\diam(\Ss^{n-1}/K)\geq\diam(\Ss^{n-1}/G)$. In the opposite direction, there is the  following result, which appears as Lemma 3.13 in \cite{Greenwald00}, and allows one to replace a group with a finite index subgroup, as long as the index is controlled:
\begin{lemma}
\label{L:index}
Let $K\subset G\subset \OO(n)$ be closed subgroups, and assume $G/K$ is finite, with $k$ elements. Then 
\[ \diam (\Ss^{n-1}/G) \geq \frac{\diam (\Ss^{n-1}/K)}{2(k-1)}.\]
\end{lemma}

The existence of a dimension-dependent lower bound on the diameter of the orbit space was established in \cite[Theorem 4.3]{Greenwald00}:
\begin{theorem}[Greenwald]
\label{T:epsilon(n)}
For each $n\geq 3$, there exists $\epsilon(n)>0$ such that, for all $G<\OO(n)$ compact and non-transitive,  $\diam(\Ss^{n-1}/G)\geq \epsilon(n)$.
\end{theorem}
See Theorem \ref{MT:finitepi1} for a generalization of Theorem \ref{T:epsilon(n)}, with similar proof.

Another result of Greenwald useful to us can be found in \cite[Theorem 3.15]{Greenwald00} and \cite[Table 1]{Greenwald00}: 
\begin{theorem}[Greenwald]
\label{T:reflection}
If $n\geq 3$ and  $G<\OO(n)$ is a finite group generated by reflections, then $\diam(\Ss^{n-1}/G)\geq \frac{\pi}{8.1}$. Moreover, if $G$ is of classical type, that is, types A, B, C or D, then $\diam(\Ss^{n-1}/G)\geq \frac{\pi}{4}$.
\end{theorem}

\section{Controlling normal subgroups via polar representations}
\label{S:polar}

We will need
the following technical observation:
\begin{lemma}
\label{L:normalizer}
Let $N_0\subset \OO(l)$ (respectively $\U(l)$, $\Sp(l)$)  be irreducible of real (respectively complex, quaternionic) type, and let $N=\Delta N_0\subset \OO(kl)$ (respectively $\U(kl)$, $\Sp(kl)$) be the diagonal group, seen as acting by left multiplication on the vector space $V$ of $l\times k$ matrices with entries in $\R$ (respectively $\C$, $\HH$). Let $\OO(k)$ (respectively $\U(k)$, $\Sp(k)$) act by right multiplication, and denote by $K$ the image in $\OO(V)$ of $\OO(k)\times \OO(l)$ (respectively of the group generated by $\U(k)\times\U(l)$ and complex conjugation, $\Sp(k)\times\Sp(l)$). Then the normalizer $N_{\OO(V)}(N)$ is contained in $K$.
\end{lemma}
\begin{proof}
Let $g\in N_{\OO(V)}(N)$. Since $g$ normalizes $N$, it also normalizes the centralizer $C_{\OO(V)}(N)$ of $N$, which, by Schur's lemma, equals $\OO(k)$ (respectively $\U(k)$, $\Sp(k)$). Therefore $g$ also normalizes $\SO(k)$ (respectively $\SU(k)$). But every automorphism of $\SO(k)$ is given by conjugation with some element of $\OO(k)$, every automorphism of $\SU(k)$ is inner, or inner composed with complex conjugation, and every automorphism of $\Sp(k)$ is inner. Therefore there exists $g'\in\OO(k)$ (respectively $\SU(k)\cup c\SU(k)$, $\Sp(k)$, where $c$ denotes complex conjugation) such that conjugation by $g$ and $g'$ coincide on $\SO(k)$ (respectively $\SU(k)$, $\Sp(k)$). In other words, $g^{-1}g'$ centralizes  $\SO(k)$ (respectively $\SU(k)$, $\Sp(k)$). By Schur's Lemma, $g^{-1}g'$ belongs to $\OO(l)$ (respectively $\U(l)$, $\Sp(l)$), and therefore $g\in K$.
\end{proof}

The next lemma is analogous to Lemma \ref{L:reducibility} in that it provides algebraic information about a representation when the diameter of the quotient is assumed to be small. In the proof we use the concept of a \emph{polar representation}, which is defined as a representation admitting a \emph{section}, that is a vector subspace which meets all of the orbits orthogonally. The quotient space of the representation is isometric to the quotient of any section by its so-called \emph{generalized Weyl group} (polar group). The latter is defined as the quotient of the subgroup which leaves the section invariant by the subgroup which fixes the section pointwise, and it is always finite. Moreover, the generalized Weyl group of a polar representation of a compact connected group is a finite reflection group. For a detailed account on polar representations and their generalized Weyl groups we refer to \cite{PT:87}.

In order to make the statement of the lemma more convenient, we make the following definition, which corresponds to the case $l=1$ in the notation of Lemma~\ref{L:normalizer}.
\begin{definition}
\label{D:super-reducible}
We will call a subgroup $N\subset \OO(V)$ \emph{super-reducible} if, as an $N$-representation, $V=W^k$, where $W$ is irreducible with $\dim_F W=1$, where $F=\R, \C$, or $\HH$ is the type of $W$.
\end{definition}

In the following we denote the symmetric group on $k$ letters by $\Sigma_k$.

\begin{lemma}
\label{L:super-reducible}
Let $G\subset \OO(n)$ be a closed subgroup, and assume $\diam (\Ss^{n-1}/G )<\pi/4$. Then, every normal subgroup $N\subset G$ is either irreducible or super-reducible.
\end{lemma}
\begin{proof}
First we claim that, as an $N$-representation, $\R^n$ has one isotypical component. Let $\R^n=V_1\oplus \cdots \oplus V_k$ be the decomposition into isotypical components, and assume to the contrary that $k>1$. Since $N$ is normal in $G$, any $g\in G$ takes $N$-invariant subspaces to $N$-invariant subspaces, and hence $N$-irreducible subspaces to $N$-irreducible subspaces. But any $N$-irreducible subspace $W$ must be contained in some $V_i$, by Schur's lemma. Thus, if $W\subset V_i$ is $N$-irreducible, and $g\in G$, there exists $j$ such that $gW\subset V_j$. Moreover, the set of $N$-irreducible subspaces of $V_i$ is connected, so that $j$ does not depend on the choice of $W\subset V_i$. So $gV_i\subset V_j$, and, applying the same argument to $g^{-1}$, it follows that $g V_i = V_j$. Thus we obtain a group homomorphism $\phi:G\to \Sigma_k$ such that $g V_i= V_{\phi(g)(i)}$ for all $i$. Since $G$ is irreducible, this action of $G$ on $\{1, \ldots, k\}$ is transitive, and, in particular, all $V_i$ have the same dimension $d$. Therefore  $G\subset \Sigma_k \ltimes \OO(d)^k$. The group $\Sigma_k \ltimes \OO(d)^k$ is polar and the quotient $\Ss^{n-1}/(\Sigma_k \ltimes \OO(d)^k)$ is isometric to the quotient of $\Ss^{k-1}$ by the  Weyl group $\Sigma_k\ltimes \{\pm 1\}^k$, which, by Theorem \ref{T:reflection}, has diameter at least $\pi/4$. Thus $\diam (\Ss^{n-1}/G)\geq \pi/4$, contradicting our hypothesis, and finishing the proof that $\R^n$ has only one $N$-isotypical component.

This puts us in the situation of Lemma \ref{L:normalizer}, and, following the notation there, $G$ is contained in $K$, which is the image in $\OO(n)$ of $\OO(k)\times \OO(l)$ (respectively of the group generated by $\U(k)\times\U(l)$ and complex conjugation, $\Sp(k)\times\Sp(l)$). If both $k,l$ are larger than one, the group $K$ is polar, non-transitive, and by direct computations the associated generalized Weyl group is a finite reflection group of classical type. Thus Theorem \ref{T:reflection} yields  $\diam (\Ss^{n-1}/G)\geq \pi/4$,
a contradiction. Therefore, either $k=1$, that is, $N$ is irreducible, or $l=1$, that is, $N$ is super-reducible.
\end{proof}

Let $G\subset \OO(n)$ be a compact subgroup with identity component $G^0$, and assume $\diam(S^{n-1}/G)<\pi/4$. Since $G^0$ is a normal subgroup of $G$, we may apply Lemma \ref{L:super-reducible} above to conclude that $G^0$ is either super-reducible, or irreducible. Thus the proof of the Main Theorem reduces to these two cases, which we will deal with separately in the next two sections.

\section{Case where $G^0$ is super-reducible}
\label{S:super-reducible}

\subsection{Finite simple groups and projective representations}
In this subsection we collect some facts about the projective representations and the automorphism groups of powers $S^r$ of a finite simple group $S$. For more details on projective representations of finite groups we refer to \cite{Karpilovsky}.

Recall that an $n$-dimensional (complex) \emph{projective representation} of a group $G$ is a group homomorphism $G\to \PGL(n,\C)$. If this homomorphism can be lifted to a group homomorphism $\rho:G\to\GL(n,\C)$, the representation is called \emph{linear}. In general, it can be lifted to a map $\rho:G\to\GL(n,\C)$ which is a group homomorphism only up to scalar multiplication. In other words, there is a map $\alpha:G\times G\to \C^\times$ such that $\rho(1)=1$, and $\rho(xy)=\alpha(x,y)\rho(x)\rho(y)$ for all $x,y\in G$. Such a map $\rho$ is called an $\alpha$-representation. The group axioms imply that $\alpha$ is a \emph{cocycle} (or \emph{Schur multiplier}), that is, it satisfies $\alpha(x,1)=\alpha(1,x)=1$ and $\alpha(x,y)\alpha(xy,z)=\alpha(y,z)\alpha(x,yz)$, for all $x,y,z\in G$. The set of all cocycles is called $Z^2(G,\C^\times)$, and it forms an Abelian group under pointwise multiplication. Moreover, one defines the subgroup $B^2(G,\C^\times)$ of \emph{coboundaries}, and the cohomology group $H^2(G,\C^\times)$ as the quotient. Two lifts of the same projective representation have cohomologous Schur multipliers, and a projective representation is linear if and only if the associated cohomology class vanishes. Let $l(G)$ denote the smallest dimension of a faithful irreducible projective representation of $G$ (if they exist, which is the case for a non-Abelian, simple group $G$).

\begin{lemma}
\label{L:l(S)}
There exists a constant $\Co$ such that, for every finite simple group $S$ that is not cyclic or alternating, one has
\[ \frac{\log|S|}{l(S)} \leq \Co. \]
\end{lemma}

See Appendix \ref{A:l(S)} for the proof, which consist of a case-by-case verification following the classification of the finite simple groups and their representations.

\begin{lemma}
\label{L:l(S)alternating}
Let $A_n$ denote the alternating group in $n$ letters. Then, for $n\geq 12$, the smallest dimension $l(A_n)$ of an irreducible faithful projective complex representation is $n-1$, uniquely achieved by the standard permutation representation on $\C^{n-1}$, and the second smallest dimension is at least $n(n-3)/4$.
\end{lemma}
\begin{proof}
Since $A_n$ is simple, every non-trivial representation is faithful. $A_n$ has exactly two cohomology classes of Schur multipliers \cite{Schur11}. Denoting by $\alpha$ the non-trivial Schur multiplier, the smallest dimension of an irreducible $\alpha$-representation is $2^{\left \lfloor{(n-2)/2}\right \rfloor }$, see \cite[page 1774]{KleshchevTiep12}. Since $n\geq 12$, this is larger than  $n(n-3)/4$, so it suffices to consider linear representations.

Every irreducible (linear) representation of $A_n$  is either the restriction of an irreducible representation of $\Sigma_n$, or a summand, with half the dimension, of such a restriction --- see \cite[page 64, Prop. 5.1]{FultonHarris}. By \cite[Result 2]{Rasala77}, when $n\geq 9$, the third smallest dimension of an irreducible representation of $\Sigma_n$ (after $1$ and $n-1$) is $n(n-3)/2$, completing the proof.
\end{proof}

\begin{lemma}
\label{L:l(S^r)}
Let $S$ be a finite simple group. If $S$ is non-Abelian, then $l(S^r)=l(S)^r$. If $S$ is Abelian, that is, $S\simeq\Z/p$ for a prime $p$, then $l(S^r)=p^{r/2}$ if $r$ is even, and $S^r$ has no complex projective faithful irreducible representations if $r$ is odd.
\end{lemma}
\begin{proof}
Assume $S$ non-Abelian. Then, by \cite[page 132, Prop. 4.1.2]{Karpilovsky}, the Schur multiplier $M(S^r)$ (that is, the cohomology group $H^2(S^r,\C^\times)$) is equal to $M(S)^r$, because of the definition of tensor product of groups in \cite[page 58]{Karpilovsky}. That is, every Schur multiplier of $S^r$ is cohomologous to a product of $r$ Schur multipliers of $S$, which, by \cite[page 198, Corollary 5.1.3]{Karpilovsky}, implies that every irreducible projective representation of $S^r$ is an outer tensor product of irreducible projective representations of $S$. Moreover, such an outer tensor product is faithful if and only if each factor is faithful, thus concluding the proof that $l(S^r)=l(S)^r$.

Next assume $S$ Abelian, that is, $S=\Z/p$, for a prime $p$, and that $S^r$ has a (complex)  faithful irreducible $\alpha$-representation for some Schur multiplier $\alpha\in Z^2(S^r,\C^\times)$. Then, by \cite[page 578, Lemma 10.4.3]{Karpilovsky}, the identity is the only element of $S^r$ that is $\alpha$-regular. Recall that, since $S^r$ is Abelian, an element $g$ is $\alpha$-regular if and only if $\alpha(g,x)=\alpha(x,g)$ for all $x\in S^r$ (see \cite[page 107]{Karpilovsky}, or \cite[Definition 1.2]{Higgs01} for the general definition of $\alpha$-regularity). Therefore, by \cite[Lemma 2.2(1)]{Higgs01}, $S^r$ is of symmetric type, which, in our case, simply means that $r$ is even; and moreover every irreducible projective $\alpha$-representation of $S^r$ has degree $\sqrt{|S^r|}=p^{r/2}$.
\end{proof}

The outer automorphism groups of the finite simple groups are small. For our purposes, the very rough estimate below will suffice:
\begin{lemma}
\label{L:auto}
Let $S$ be a finite simple group, and $r\geq 1$. If $S$ is non-Abelian, then $|\Out(S)|\leq |S|$, $|\operatorname{Aut}(S)|\leq |S|^2$, and $|\operatorname{Aut}(S^r)|\leq r!|S|^{2r}$. If $S$ is Abelian, $S=\Z/p$ for a prime $p$, then $|\operatorname{Aut}(S^r)|\leq p^{r^2}$.
\end{lemma}
\begin{proof}
Assume $S$ non-Abelian. Using the classification of finite simple groups, it has been proved in \cite[Lemma 2.2]{Quick04} that $|\Out(S)|\leq |S|/30$. Since $\operatorname{Inn}(S)\simeq S$, it follows that $|\operatorname{Aut}(S)|\leq |S|^2$.
Moreover, $\operatorname{Aut}(S^r)$ is isomorphic to the semi-direct product between the permutation group in $r$ letters, and $\operatorname{Aut}(S)^r$. Indeed,  any group homomorphism $\phi:S^r\to S^r$ can be written as
\[ \phi(g_1, \ldots, g_r)= \left(\prod_j \phi_{1j}(g_j), \prod_j \phi_{2j}(g_j), \ldots,  \prod_j \phi_{rj}(g_j) \right)\]
where $\phi_{ij}:S\to S$ are group homomorphisms such that, for all $i$, and all $(g_1, \ldots, g_r)\in S^r$, $\{\phi_{ij}(g_j)\}_{j=1}^r$ commute. Since $S$ is simple non-Abelian, this implies that for each $i$, there is at most one value of $j$, such that $\phi_{ij}$ is non-trivial (and hence an automorphism). Assuming further that $\phi$ is an automorphism, there must in fact be a permutation $\sigma\in \Sigma_r$ such that $\phi_{ij}$ is non-trivial if and only if $\sigma(i)=j$.

In the Abelian case, an automorphism of $S^r$ is represented by an $r\times r$ matrix with entries in $\Z/p$, and thus $|\operatorname{Aut}(S^r)|\leq p^{r^2}$.
\end{proof}

\subsection{Volume, diameter and dimension}
 We will need a rough estimate for the volume of
	 the compact rank one-symmetric spaces (which is actually  known explicitly).
	Denote by $B^n$ the unit Euclidean ball and  by $\C P^n =\Ss ^{2n+1} /\U(1)$ and $\HH P^n =\Ss^{4n+3} / \Sp (1)$ the complex and the quaternionic projective spaces, respectively. Note that 
	$\C P^n$ and $\HH P^n$ equipped with their canonical,
          quotient
        metrics have sectional curvatures bounded above by $4$.

\begin{lemma}
Let  $M$ be an $n$-dimensional compact, simply connected rank one symmetric space with curvature bounded above by $4$. Then
$$\vol (M)  > {\frac 1 {2^n}} \vol (B^n) \,.$$
\end{lemma}
\begin{proof}
	The injectivity radius of the symmetric space $M$ is at least equal to the injectivity radius of the sphere $\frac 1 2 \Ss ^n$  of constant curvature $4$.  By the Bishop--Gromov volume comparison, we have
	$$\vol (M) \geq \vol (\frac 1 2 \Ss^n) \,.$$
	Considering the orthogonal projection, the volume of $\frac 1 2 \Ss ^n$ is larger  than the volume of the unit $n$-dimensional Euclidean ball of radius $\frac 1 2$. This implies the claim.
\end{proof}	
	
	As an application we deduce:

\begin{lemma}
\label{L:volume}
Let $M$ be an $n$-dimensional  compact, simply connected rank one  symmetric space
with curvature bounded above by $4$.
Let $G$ be a finite group acting by isometries on $M$.
 Then $$\diam (M/G) \geq \frac 1 2  \sqrt[n]{1/|G|} \, . $$
\end{lemma}
\begin{proof}
Let $d$ denote the diameter of the quotient. Then a fundamental domain for the action is contained in a ball in $M$ of radius $d$. Since $M$ is positively curved, by the Bishop--Gromov Theorem, the volume of the quotient satisfies $d^n\vol(B^n)\geq \vol (M^n/G)$. 
On the other hand, $\vol (M/G) \geq \vol(M)/|G|$ (with equality if the action is effective). 
 The result now follows from
the previous Lemma.
\end{proof}


\subsection{Proof of Main Theorem --- super-reducibly case}
We can now prove the Main Theorem under the assumption that 
the connected component $G^0$ of $G$ acts as scalars, i.e. \emph{super-reducibly} in terms of Definition \ref{D:super-reducible}).

\begin{theorem}
\label{T:almostfinite}
There exists $\epsilon>0$ such that $\diam(\Ss^{n-1}/G)>\epsilon$ for every  group $G\subset \OO(n)$, for which the connected component $G^0 \subset \OO(n)$ is super-reducible. 
\end{theorem}

We start with a few reductions:

\begin{lemma}
\label{L:assumptions}
To prove Theorem \ref{T:almostfinite}, one may assume that  $\diam(\Ss^{n-1}/G)<\pi/4$ and $n>16$.  Moreover, we may assume that one of the following three cases occurs:
\begin{enumerate}
\item The group $G$ is finite,  $\OO (1) =\{\pm 1\}\subset G$ and $\{\pm 1\}$ is maximal among super-reducible normal subgroups of $G$.
\item We have $G=H\cdot \Sp (1)$. The group $\Sp(1)= G^0$ is  maximal among super-reducible normal subgroups of $G$. The group $H$ is finite and contains the center $Z$ of  $\Sp(n/4)$.
\item We have $G=H\cdot \Sp (1)$. The group $\U(1)= G^0$ is  maximal among super-reducible normal subgroups of $G$. The group $H$ is finite and contains the center $Z$ of  $\SU(n/2)$.
\end{enumerate}
\end{lemma}
\begin{proof}
 The first statement is clear. By Theorem
 \ref{T:epsilon(n)}  we may assume  $n>16$.

  Consider a subgroup $L$ of $\OO (n)$ which is maximal among subgroups that are super-reducible, contain $G^0$ and are normalized by $G$. Replacing $G$ by $G\cdot L$
  and observing that the connected component of $G\cdot L$ is the super-reducible group $L^0$, we may assume that $G=G\cdot L$, hence $L\subset G$.
  
   Clearly, $\pm 1 \in L$.  If $L= \{\pm  1 \}$ we are in case (1).
   
   Otherwise, $L$ is of complex or quaternionic type. Assume that $L$ is of quaternionic type, hence $L\subset  \Sp (1)$.  By  Lemma \ref{L:normalizer}, the group $G$ is contained  in $\Sp (1) \cdot \Sp (n/4)$.   Hence $\Sp (1)$ is normalized by $G$, thus $L=\Sp (1)$, by maximality of $L$.  We can now take $H$ to be  the intersection of $G$ with   $\Sp(n/4)$.



Similarly, if $L$ is of complex type then applying Lemma \ref{L:normalizer} and the maximality of $L$ we obtain $L=\U (1)$ and
$G$ is contained in the extension of $\U (n/2)$ by the complex-conjugation. Replacing $G$ by an index two subgroup, which is possible by Lemma   \ref{L:index}, we may assume that $G\subset \U (n)$.  Again, we obtain $H$  as the intersection of $G$ with $\SU (n/2)$.
\end{proof}

\begin{proof}[Proof of Theorem \ref{T:almostfinite}]
We make the assumptions listed in Lemma \ref{L:assumptions}. If $G$ is finite, we set $H=G$ to make the notation more uniform. We denote by $Z$ the center of $\OO(n)$ (resp. $\SU(n/2)$, $\Sp(n/4)$), so $Z$ is cyclic of order $2$ (resp. $n/2$, $2$).
Let $\bar{N}$ be a minimal normal subgroup of $H/Z$.  Since $\bar{N}$ is minimal normal, it is characteristically simple, hence isomorphic to $S^r$, for some finite simple group $S$ (see \cite[Lemmas 2.7 and 2.8]{Wilson}).

We will show that $\log(|H/Z|)/n$ is uniformly bounded from above by providing appropriate bounds on $n$ and on $|H/Z|$. This will conclude the proof via an application of Lemma \ref{L:volume}, because $\Ss^{n-1}/G=\C P^{(n-2)/2}/(H/Z)$ (resp. $\HH P^{(n-4)/4}/(H/Z)$).

Let $N$ be the inverse image of $\bar{N}$ in $H$.
We claim that $N\subset \OO(n)$ is irreducible.
Indeed, $N\cdot\{\pm 1\}$ (resp. $N\cdot\U(1)$, $N\cdot\Sp(1)$) must be irreducible, because it is normal in $G$, and strictly contains $\{\pm 1\}$ (resp. $\U(1)$, $\Sp(1)$), which is maximal super-reducible by assumption.
This implies that, as an $N$-representation, $\R^n$ breaks into at most $1$ (resp. $2$, $4$) irreducible factors. Since $n>16$, $N$ cannot be super-reducible, and since it is normal in $G$, it must be irreducible. 

Next, we claim that the centralizer $C_H(N)$ of $N$ in $H$ is $Z$, so that, in particular, $Z$ is the center of $N$. Indeed, since $N$ is irreducible, $C_H(N)$ is super-reducible. This implies that $C_H(N)$ (resp. $C_H(N)\cdot\U(1)$, $C_H(N)\cdot\Sp(1)$) is not irreducible, because $n>16$. Thus, being normal in $G$, it must be super-reducible. By maximality of $\{\pm 1\}$ (resp. $\U(1)$, $\Sp(1)$) among super-reducible normal subgroups of $G$, we must have $C_H(N)=\{\pm 1\}$ (resp. $C_H(N)\cdot\U(1)=\U(1)$, $C_H(N)\cdot\Sp(1)=\Sp(1)$), which implies $C_H(N)=Z$.

{\bf Bounding $|H/Z|$ from above}. $H$ acts by conjugation on $N$, so we have a group homomorphism $\eta: H\to\Aut(N)$, whose kernel is $C_H(N)=Z$. Thus $|H|\leq |Z|\cdot |\operatorname{image}(\eta)|$. But
\[ \operatorname{image}(\eta) \subset \Aut_Z(N)=\{\phi\in\Aut(N)\ |\ \phi(z)=z\ \ \forall z\in Z\}\]
and each element of $\Aut_Z(N)$ induces an automorphism of $\bar{N}=N/Z$. Thus, denoting $\Aut_0(N)=\{\phi\in \Aut_Z(N)\ |\ \phi \text{ induces the trivial automorphism of }\bar{N}\}$, we have a short exact sequence
\[ 1\to \Aut_0(N) \to \Aut_Z(N) \to \Aut(\bar{N})\to 1. \]
Moreover, the map that sends $\phi\in\Aut_0(N)$ to $\alpha:\bar{N}\to Z$ defined by $\alpha(x)=\phi(x) x^{-1}$ establishes an isomorphism $\Aut_0(N)\simeq\Hom(\bar{N},Z)$. If $S$ is non-Abelian, $\Hom(\bar{N},Z)$ is trivial, and if $S$ is Abelian, we have $|\Hom(\bar{N},Z)|\leq |Z|^r$. Therefore, we may use Lemma \ref{L:auto} to obtain the bound
\begin{equation}
\label{E:Hbound}
|H/Z| \leq \left\{
\begin{array}{rl}
r!|S|^{2r} & \text{if }S\text{ is non-Abelian}\\
n^{r}p^{r^2} & \text{if }S=\Z/p
\end{array} \right.
\end{equation}

{\bf Bounding $n$ from below.} Consider the representation of $N$ on $\R^n$. It is faithful and irreducible. If it is of complex or quaternionic type (that is, if it commutes with some complex structure) then  $U=\R^n=\C^{n/2}$ is a faithful irreducible complex representation of $N$. Otherwise, the representation of $N$ on $\R^n$ is of real type, so that its complexification $U=\C^n$ is a faithful irreducible complex $N$-representation. 

The projectivization of $U$ has kernel which must be equal to $Z$, because $Z$ is the center of $N$. Thus we have obtained a projective faithful irreducible representation of $\bar{N}=S^r$ of dimension $n$ or $n/2$, and thus, via Lemma \ref{L:l(S^r)}, the bound
\begin{equation}
\label{E:nbound}
n \geq \left\{
\begin{array}{rl}
l(S^r)=l(S)^r & \text{if }S\text{ is non-Abelian}\\
p^{r/2} & \text{if }S=\Z/p.
\end{array} \right.
\end{equation}

To show that $\log|H/Z|/n$ is uniformly bounded and conclude the proof, we divide into three cases: $S$ Abelian, $S$ non-Abelian and non-alternating, and $S$ alternating and non-Abelian.

If $S$ is Abelian, isomorphic to $\Z/p$, then from \eqref{E:Hbound} and \eqref{E:nbound} we obtain
\begin{equation}
\label{E:Abelian}
\frac{\log|H/Z|}{n}\leq \frac{r\log n +r^2 \log p}{n}\leq 
\frac{r }{p^{r/4}}.\frac{\log(n)}{\sqrt{n}}+\frac{r^2}{p^{r/4}}.\frac{\log p}{p^{r/4}}
\end{equation}
which is bounded from above.

If $S$ is non-Abelian, then \eqref{E:Hbound} and \eqref{E:nbound} yield
\begin{equation}
\label{E:non-Abelian}
\frac{\log|H/Z|}{n}\leq \frac{\log(r!) + 2r\log|S|}{n}\leq \frac{\log(r!)}{l(S)^r}+\frac{2r\log|S|}{l(S)^r}
\end{equation}
Since $l(S)\geq 2$, the term $\frac{\log(r!)}{l(S)^r}$ is bounded. When $S$ is non-alternating, the last term $\frac{2r\log|S|}{l(S)^r}$ is bounded, because, by Lemma \ref{L:l(S)}, the quantity $\frac{\log|S|}{l(S)}$ is bounded. From now on assume $S$ is the alternating group $A_d$. If $r\geq 2$ and $d\geq 12$, then using Lemma \ref{L:l(S)alternating} we see that the last term in \eqref{E:non-Abelian} is again bounded:
\begin{equation}
\label{E:alternating}
\frac{2r\log|S|}{l(S)^r}\leq \frac{2r\log(d!/2)}{(d-1)^r}\leq
 \frac{2r}{(d-1)^{r-3/2}}.\frac{d\log d}{(d-1)^{3/2}}.
 \end{equation}
If $r\geq 2$ and $5 \leq d< 12$ then
\begin{equation}
\label{E:alternating2}
\frac{2r\log|S|}{l(S)^r}\leq \frac{2r\log(12!/2)}{2^r}
 \end{equation}
which is bounded. Thus we may assume $r=1$. If the faithful irreducible projective representation of $S=A_d$ constructed above is not the standard permutation representation, then by Lemma \ref{L:l(S)alternating} its dimension is at least $d(d-3)/4$, so that $\frac{2r\log|S|}{n}$ is again bounded. 

Therefore we have reduced to the case where $S=A_d$ and the projective representation $U$ of $\bar{N}=S=A_d$ constructed above is the standard representation on $\C^{d-1}$. Since this projective representation $A_d=\bar{N}\to \PGL(U)$ lifts to the linear representation $A_d\to \GL(U)$, the short exact sequence $1\to Z \to N \to \bar{N} \to 1$ splits, which implies that $N\simeq A_d\times Z$ (because $Z$ is the center of $N$). Thus  $A_d$ is a normal subgroup of $G$, and it acts in the standard way on $U=\C^{d-1}$. Recall that $U=\C^{d-1}$ was  isomorphic to either $\R^n$, or its complexification. The first case is precluded by our hypotheses, since then the restriction of the $G$-representation $\R^n$ to $A_d$ would be neither irreducible nor super-reducible. Therefore the subgroup $A_d$ acts on $\R^n=\R^{d-1}$ in the standard way. Since this representation is of real type, that is, it does not leave any complex structure invariant,  we are in the case where $G=H$ is finite, so, in particular, $Z=\pm 1$. Since the automorphism group of $A_d$ is isomorphic to $\Sigma_d$ (because $d\geq 7$, see \cite[Theorem 2.3]{Wilson}), the index of $A_d$ in $G$ is at most four, and the desired diameter bound follows from Theorem \ref{T:reflection} and Lemma \ref{L:index}.
\end{proof}

\section{Case where $G^0$ is irreducible}
\label{S:irreducible}

As noted at the end of Section \ref{S:polar}, Lemma \ref{L:super-reducible} reduces the proof of the Main Theorem to two cases, according to whether the identity component $G^0$ acts as scalar multiplication or irreducibly. We have dealt with the former in Section \ref{S:super-reducible}, and this section is devoted to the latter.

For convenience in this section we will consider \emph{almost} faithful representations $\rho:G \to \OO(V)$. We first lift $\rho$ to a representation of a semidirect product.
By~\cite[Lemma~7.5]{Wilking99}, there is a finite subgroup $\Gamma$ meeting all
connected components of $G$. Since the identity component
$G^0$ is a normal subgroup of $G$,
we can write $G=G^0\cdot\Gamma$. Now there is a finite covering
$G^0 \rtimes \Gamma\to G$ and we can lift $\rho$ to the semidirect product. 
Therefore from now on we assume $G$ splits as $G^0 \rtimes \Gamma$.

Furthermore, by passing to a finite cover we may also assume that $G^0$ is a product of simply connected simple Lie groups, and a torus. Altogether, we have reduced to proving the following:

\begin{theorem}
\label{T:irreducible}
There exists $\epsilon>0$ with the following property. Let $G^0$ be a product of a torus with finitely many simply-connected compact connected
simple Lie groups, let $\Gamma$ be a finite group acting on $G^0$ by automorphisms, and set $G=G^0 \rtimes \Gamma$. Let $\rho:G\to\OO(V)$ be an almost faithful representation whose restriction to $G^0$ is irreducible but not transitive. Then $\diam(\Ss(V)/G)>\epsilon$.
\end{theorem}

Our strategy to prove Theorem \ref{T:irreducible} is the following reduction:
\begin{lemma}
\label{L:G^0}
To prove Theorem \ref{T:irreducible}, it suffices to find a common lower diameter bound for all non-transitive representations of the following types:
\begin{enumerate}
\item[(i)] $G=G'$ is a simply-connected compact connected
  simple Lie group, $V=V'$ is of real type.
\item[(ii)] $G=\U(1)\times G'$, $V=\C\otimes_{\mathbb C}V'$,
$G'$ is a simply-connected compact connected simple Lie group and 
$V'$ is of complex type.
\item[(iii)] $G=\Sp(1)\times G'$, $V=\HH\otimes_{\mathbb H}V'$,
$G'$ is a simply-connected compact connected simple Lie group and 
$V'$ is of quaternionic type. 
\item[(iv)] $G=\Sigma_k\ltimes\SO(n)^k$,
$V=\otimes^k\R^n$, where $n\geq3$ and $k>2$. 
\item[(v)] $G=\U(1)\times \Sigma_k\ltimes\SU(n)^k$,
$V=\C\otimes_{\mathbb C}\otimes^k\C^n$, where $n\geq3$ and $k>2$.
\item[(vi)] $G=\Sigma_k\ltimes \Sp(n)^k$, $V=\otimes^k\HH^n$,
where $n\geq1$ and $k\geq4$ is even. 
\item[(vii)] $G=\Sp(1)\times \Sigma_k\ltimes \Sp(n) ^k$, $V=\HH\otimes_{\mathbb H}\otimes^k\HH^n$,
where $n\geq1$ and $k\geq3$ is odd. 
\end{enumerate}
(in the last four cases, the permutation group $\Sigma_k$ acts by permuting the factors of the tensor product)
\end{lemma}

\begin{remark}\label{R:epsilons}
  Some explanation about quaternionic tensor products
  is in order for cases (vi) and (vii) above. The complex representation of $\Sp(n)$ on $\HH^n=\C^{2n}$ is of quaternionic type, and thus $W=\otimes_\C^k(\HH^n)$ is a complex irreducible representation of $\Sp(n)^k$, which is of real type when $k$ is even, and of quaternionic type when $k$ is odd. Denoting by $\epsilon_i:\HH^n\to\HH^n$ the standard quaternionic structure on the $i$th factor of this tensor product, $\epsilon=\epsilon_1\otimes\cdots\otimes\epsilon_k$ is a real (if $k$ even) or quaternionic (if $k$ odd) structure on $W$. In case (vi) we take $V$ to be a real form of $W$, that is, the fixed point set of $\epsilon:W\to W$. In case (vii) we take $V$ to be the real form of $\C^2\otimes_\C W$ relative to $\epsilon_0\otimes\epsilon$, where $\epsilon_0$ is the standard quaternionic structure on $\HH=\C^2$. In both cases the permutation group $\Sigma_k$ acts on $W$ by permuting the factors, and this action commutes with $\epsilon$, so that it induces an action on $V$.
\end{remark}

The proof of Lemma \ref{L:G^0} is obtained by analysing the action by $G^0$ using Lemmas \ref{L:normalizer} and \ref{L:super-reducible},
and is relegated to Appendix \ref{A:G^0} (alternatively, one may also note that every maximal closed \emph{non-transitive} subgroup of $\OO(n)$ (up to taking subgroup of small index) is either a maximal closed subgroup of $\OO(n)$, or $\U(1)$ times a maximal closed subgroup of $\SU(n)$, or $\Sp(1)$ times a maximal closed subgroup of $\Sp(n)$, and then use the classification of infinite, non-simple  maximal closed subgroups of the classical groups obtained in \cite{AFG12}).
In the remaining of this section, we run through the cases of
Lemma~\ref{L:G^0}. 

\subsection{The tensor power representations} \label{SS:puretensors}
The goal of this subsection is to show the existence of a universal lower bound on $\diam \Ss(V)/G$, where $(V,G)$ is one of the representations listed in cases (iv)-(vii) of Lemma \ref{L:G^0}.

For an arbitrary metric space $X$, define the
\emph{radius} at $x\in X$ to be
$r_x = \inf\{r > 0 : X \subset B(x,r)\}$.
It is immediate from the triangle inequality that it compares
to the diameter of $X$ as follows:
\begin{equation}\label{E:radius}
  r_x \leq \diam\,X \leq 2r_x.
  \end{equation}

\begin{lemma}\label{L:fixedpoint}
  Let $G$ be a locally compact topological group
  acting continuously, properly and isometrically
  on a metric space $X$. Assume the fixed point set of $G$ on $X$ is
  non-empty. Then $\diam\,X/G\geq\frac12\diam\,X$. 
\end{lemma}

\begin{proof}
Let $x_0\in X$ be a fixed point of $G$ and denote by $\pi:X\to X/G$
the natural projection. For every $x\in X$, the distance from $x_0$
to $Gx$ is constant. It follows that the distances $d(x,x_0)=d(\pi(x),\pi(x_0))$
and hence the radii $r_{x_0}=r_{\pi(x_0)}$. The desired result now follows
from~(\ref{E:radius}).
\end{proof}

Next consider a representation
$\rho:G\to\OO(V)$ as in the last four cases of Lemma \ref{L:G^0}. 
Then $G$ and $G^0$ share a common orbit in $V$, namely, that one consisting
of ``pure tensors''. Indeed, following the notation in Remark \ref{R:epsilons},
it is the orbit through 
 $p=v_1\otimes\cdots\otimes v_k\in\otimes^k\R^n$ in case (iv),
$p=v_1\otimes\cdots\otimes v_k\in\otimes^k\C^n$ in case (v),
$p=v_1\otimes\cdots\otimes v_k+\epsilon_1v_1\otimes\cdots\otimes\epsilon_kv_k
\in\otimes^k\C^{2n}$ in case (vi) and 
$p=v_0\otimes\cdots\otimes v_k+\epsilon_0v_0\otimes\cdots\otimes\epsilon_kv_k
\in\C^2\otimes_{\mathbb C}\otimes^k\C^{2n}$ in case (vii).

Denoting $X=\Ss(V)/G$
and $X^0=\Ss(V)/G^0$, we conclude that $G/G^0$ acts on $X^0$ with a fixed point
and $X=X^0/\Gamma$, so we can apply
Lemma~\ref{L:fixedpoint} and Lemma \ref{L:super-reducible} to deduce that
$\diam X\geq \frac12\diam X^0\geq\pi/8$.

\subsection{Normal injectivity radius and focal radius}\label{inj-rad}
This subsection is devoted to proving a version of the injectivity radius estimate of Klingenberg for quotients $\Ss(V)/G$, that is, to give a lower bound for the normal injectivity radius of a $G$-orbit in terms of the focal radius. In the next section this will be combined with a universal lower bound (found in \cite{GorodskiSaturnino17}) for the focal radius for a special $G$-orbit  to finish the proof of Theorem \ref{T:irreducible}.

Let $N$ be a properly embedded submanifold of a complete Riemannian
manifold $M$. Consider the normal bundle $\nu N$ in $M$ and
the normal exponential map $\exp^\perp:\nu N\to M$. Denote the
open ball bundle of radius $r$ in $\nu N$ by $\nu^rN$. The
\emph{normal injectivity radius} $\iota_N$
of $N$ is the
the supremum of the numbers~$r$ such that $\exp^\perp$ is an embedding
on $\nu^rN$, and the image of $\nu^{\iota_N}N$ is called
the \emph{maximal tubular neighborhood} of $N$. If $N$ is compact,  $\iota_N>0$.
On the other hand, a \emph{focal point} of $N$ relative to $p\in N$
is a critical value of $\exp^\perp:\nu N\to M$ such that $\exp^\perp(v)=q$
for some $v\in\nu_pN$. In this case, the \emph{focal distance}
associated to~$q$ is the length $|v|$ of the normal geodesic from $p$
to $q$. The \emph{focal radius} $f_N$ of $N$ is the
infimum of all focal distances to $N$ along normal geodesics.
It is clear that $\iota_N\leq f_N$.

\begin{proposition}\label{P:Klingenberg}
  Let $G$ be a compact Lie group acting isometrically
  on a compact Riemannian manifold $M$. Let $p\in M$ and
  consider the orbit $N=Gp$. Assume the fixed point set of the identity
  component $(G_p)^0$ of the isotropy group at~$p$
  in the closure of the maximal tubular neighborhood of $N$
  is contained in $N$. Then
  $f_N/2\leq\iota_N$. In particular, the diameter of $M/G$ is bounded
  below by $f_N/2$.
\end{proposition}

\begin{proof}
If $\iota_N< f_N$, we argue as in Klingenberg's Lemma \cite[Chap.~13 Proposition~2.12]{doCarmo} (see also \cite[Lemma~5.6]{CheegerEbin})
to deduce the existence of a horizontal geodesic segment $\gamma$
of length $2\iota_N$, entirely contained in the closure of
the maximal tubular neighborhood of $N$, 
that starts at $p$ and ends at a point $q\in N$. By assumption,
$\Lg_p$ is not contained in $\Lg_{\gamma(t)}$ for all small $t>0$.
Therefore there is a non-trivial variation of $\gamma$
through horizontal geodesics fixing $p$, and ending on $N$,
and hence $p$ is a focal point of $N$. It follows that the length of
$\gamma$ is at least $f_N$, as desired.
Finally, any point in $M$ outside the maximal tubular neighborhood of $N$
has distance at least $\iota_N$ \todo{replaced $i_N$ with $\iota_N$} to $N$, which proves the last statement.  
\end{proof}

\subsection{The case of simple Lie groups and their extensions by scalars}\label{i-ii-iii}
Having dealt with cases (iv)--(vii) of Lemma \ref{L:G^0} in Subsection \ref{SS:puretensors}, it remains to treat cases (i)--(iii) to finish the proof of Theorem \ref{T:irreducible}, and hence of the Main Theorem. 
The strategy to prove the existence of a lower diameter bound for $\Ss(V)/G$ for the representations listed in cases (i)--(iii) of Lemma \ref{L:G^0} is to use the universal lower bound for the focal radius of a special orbit \cite{GorodskiSaturnino17} in combination with the Klingenberg-type Proposition \ref{P:Klingenberg}. In fact, this subsection is devoted to showing that, in these cases, the hypothesis in Proposition \ref{P:Klingenberg} concerning
the fixed point set is satisfied.

Fix a maximal torus of $G$, consider
the corresponding root system and fix an ordering of the roots.
In view of Theorem~\ref{T:reflection}, we may assume that
the representation of $G$ is not polar. \todo{added non-polar assumption}
Since it is also irreducible and non-transitive on 
the unit sphere, we may apply the main result of \cite{GorodskiSaturnino17} to 
deduce there exists $\delta>0$ such that
the focal radius $f_N$ of the orbit $N=Gp$ is bigger than $\delta$,
where $p=v_\lambda$ or $p=\frac1{\sqrt2}(v_\lambda+\epsilon(v_\lambda))$, 
and $v_\lambda$ is a unit highest weight vector of $V$ or its 
complexification~$V^c$, according to whether
$\rho$ admits an invariant complex structure or not;
in the latter case it admits a real structure $\epsilon$. 
Note that $\rho$ admits an invariant complex structure in case~(ii)
and it does not in cases~(i) and~(iii).

In case (ii), we have $v_\lambda=v_{\lambda'}$ is also a unit highest
weight vector of $V'$. 
In case (iii), $V'$ admits an
invariant complex structure and it is easier to do the computations in $V'$;
let $v_{\lambda'}$ be a unit highest weight vector.
The $G'$-action on $V'$ admits an extension to a $G$-action. 
We have $V=\HH\otimes_{\HH}V'$ is a real form of
$\C^2\otimes_{\C}V'$ and there is a $G$-equivariant isometry
$V'\to V$ mapping $v_{\lambda'}$ to $\frac1{\sqrt2}(v_\lambda+v_{-\lambda})$,
where $v_\lambda$ is the highest weight vector of $V$
and $v_{-\lambda}=\epsilon(v_\lambda)$, where $\epsilon$ is the
real structure on $\C^2\otimes_{\C}V'$.

\subsubsection{The complex and quaternionic cases}\label{cq}

In this section, we check the hypothesis of
 Proposition \ref{P:Klingenberg} in cases~(ii) and~(iii).

\begin{lemma}
\label{L:fix-pt-ii-iii}
  Let $p=v_{\lambda'}$ in cases~(ii) and~(iii).
  Then the fixed point set of $G_p^0$ in $\Ss(V')$ is
  contained in $Gp$.
\end{lemma}

\begin{proof}
The proof is the same in both cases. The Lie algebra of the maximal torus
of $G$ has the form $\Lt'\oplus\mathfrak{u}(1)$ where $\Lt'$ is the
Lie algebra of the maximal torus of $G'$. The isotropy algebra $\Lg_p$
contains the kernel of $\lambda$ in $\Lt'$ and an element of the form
$h_1-h_0$, where $h_1\in\Lt'$ satisfies $\lambda(h_1)=i$
and $h_0\in\mathfrak{u}(1)$ acts as
multiplication by $i$ on $V$. Write an arbitrary element of $\Ss(V')$ as
$v=\sum_\mu c_\mu v_\mu$, where the sum runs through the different weights of
$V'$, $c_\mu\in\C$ and $v_\mu$ is a weight vector of weight $\mu$.
Then $\ker\lambda|_{\Lt'}\cdot v=0$ implies $c_\mu=0$ unless $\mu$ is a
multiple of $\lambda'$. Moreover, if $\mu = c\lambda'$ then
$(h_1-h_0)\cdot v_\mu= (ci-i)v_\mu$ can be zero only if $c=1$.
It follows that $\Lg_p\cdot v=0$ implies $v=c_{\lambda'}v_{\lambda'}$ with
$|c_{\lambda'}|=1$, so $v\in Gp$.
\end{proof}

\subsubsection{The real case}\label{r}

It remains to tackle case~(i) from Lemma \ref{L:G^0}.
\todo{removed non-polar assumption} We claim that we may also assume
that $\Lg_p$ is a maximal isotropy algebra, up to conjugation.
Indeed, let $q\in \Ss(V)\setminus\{-p\}$ be arbitrary, consider the minimal
geodesic segment $\gamma$ in $\Ss(V)$
from $p$ to the orbit $Gq$ and let $q_1\in Gq$
be its endpoint. Of course, $\Lg_{q_1}$ and $\Lg_q$ are
$\mathrm{Ad}_G$-conjugate. If $\Lg_{q_1}$ is not contained in $\Lg_p$,
an element in $\Lg_{q_1}\setminus\Lg_p$ produces a non-trivial variation of
$\gamma$ through horizontal geodesics fixing~$q_1$, which implies that
$q_1$ is a focal point of $Gp$. We deduce that
$\diam\, X\geq\ell \geq f_N> \delta$,
where $\ell$ is the length of $\gamma$ and $N=Gp$.

So in the sequel we may assume
$\Lg_p$ is a maximal isotropy algebra, up to conjugation.
We will show that this implies that $\mathrm{rk}\,\Lg\geq2$
and $V^c$ is a \emph{minuscule}
representation, that is, all weights comprise a 
single Weyl orbit.

\begin{lemma}\label{L:rank}
  $\mathrm{rk}\,\Lg_p=\mathrm{rk}\,\Lg-1$.
\end{lemma}
\begin{proof}
  Denote the Lie algebra of the maximal torus of $G$ by $\Lt$,
  and the corresponding system of roots by $\Delta$, where we have
  already chosen an ordering of the roots.
  Consider the root space decomposition
\[  \Lg=\Lt+\Lt^\perp,\qquad \Lt^\perp=\sum_{\alpha\in\Delta^+}(\Lg^{\mathbb C}_\alpha+\Lg^{\mathbb C}_{-\alpha})\cap\Lg
\]
It is clear that $\Lg_p=\Lg_p\cap\Lt+\Lg_p\cap\Lt^\perp$, where
$\Lg_p\cap\Lt=\ker\lambda$. Suppose, to the
contrary, that $\mathrm{rk}\,\Lg_p=\mathrm{rk}\,\Lg$. Then $\ker\lambda$
can be enlarged to a Cartan subalgebra of $\Lg_p$ by adding
an element $u$ of $\Lt^\perp$. It follows from $[\ker\lambda,u]=0$
that $u=x_\alpha+\epsilon x_\alpha$ for some $\alpha\in\Delta^+$,
where $x_\alpha\in\Lg^\C_\alpha$, and $\lambda$ is a multiple of $\alpha$.
Since $0=u\cdot (v_\lambda+v_{-\lambda})$, we deduce that $-\lambda+\alpha=\lambda-\alpha$ and thus $\lambda=\alpha$.
The only dominant roots
  are the highest root or the highest short root. 
  In the first case, our representation is the adjoint representation and hence
  polar. The remaining cases that need to be analyzed occur only
  for simple groups of type $B_n$, $C_n$, $F_4$, $G_2$, in which
  our representation is respectively the isotropy representation
  of the symmetric space $\Ss^{2n+1}$, $\SU(2n)/\Sp(n)$,
  $\operatorname{E}_6/\operatorname{F}_4$
  or the $7$-dimensional
  representation of $\operatorname{G}_2$, again all polar.
  In any case, we reach a 
contradiction to our previous assumption.
\end{proof}

\begin{lemma} $\mathrm{rk}\Lg\geq2$
and $V^c$ is \emph{minuscule}.
\end{lemma}
\begin{proof}
It follows from Lemma~\ref{L:rank} that zero cannot be a weight,
because the isotropy algebra of a real zero-weight vector
would have full rank in $\Lg$, bigger than $\mathrm{rk}\,\Lg_p$.
This already rules out representations of real type of a rank one 
group, since the odd dimensional representations of $\SO(3)$
always have zero as a weight.

Take $q=\frac1{\sqrt2}(v_\mu+v_{-\mu})$
where $\mu$ is an arbitrary nonzero weight $\mu$ of~$V^c$.  Then
$\ker\mu \subset \Lg_q$ and
$\mathrm{rk}\,\Lg_q\leq\mathrm{rk}\,\Lg_p$. Again Lemma~\ref{L:rank}
implies
that $\mathrm{rk}\,\Lg_q=\mathrm{rk}\,\Lg_p$ and $\ker\mu$, $\ker\lambda$,
viewed as subspaces of $\Lt$, are $\mathrm{Ad}$-conjugate.
Since two maximal tori of a compact connected Lie group
are $\mathrm{Ad}$-conjugate by a transformation that fixes
pointwise their intersection, we deduce that 
$\ker\mu$, $\ker\lambda$ are conjugate under the Weyl group $W$.
Now $\mu$ is $W$-conjugate to a multiple of $\lambda$, say $c\cdot\lambda$
with $0<c\leq1$.

Since $\mathrm{rk}\,\Lg\geq2$, we can find
a simple root $\alpha$ of $\Lg$ which is neither 
proportional nor orthogonal to $\lambda$.
Then $s_\alpha\lambda:=\lambda-2\frac{\langle\lambda,\alpha\rangle}{||\alpha||^2}\alpha$
is a weight of $V^c$ and so are $\lambda$, $\lambda-\alpha,\ldots,\lambda-q\alpha$
where $q=2\frac{\langle\lambda,\alpha\rangle}{||\alpha||^2}$ is a positive
integer.
These weights are all $W$-conjugate to a multiple of $\lambda$ 
by what we have seen above, and they
all lie in the union of two closed chambers because $\alpha$ is simple. 
Since $W$ acts
transitively on the set of chambers, we deduce that $q=1$. In particular,
there can be no weights of $V^c$ of the form $c\cdot\lambda$, $0<c<1$.
We have proved that all non-zero weights of $V^c$ are $W$-conjugate.
Therefore,
$V^c$ is minuscule.
\end{proof}

We finally check the hypothesis of
 Proposition \ref{P:Klingenberg}.

\begin{lemma}\label{fix-pt2}
  Let $p=\frac1{\sqrt2}(v_\lambda+v_{-\lambda})$.
  Then the fixed point set of $G_p^0$ in $\Ss(V)$ is $\{\pm p\}$. 
\end{lemma}  
\begin{proof}
Recall zero is not a weight of $V^c$. 
Write an arbitrary element of $\Ss(V)$ as
$v=\sum_\mu c_\mu (v_\mu+v_{-\mu})$, where $c_\mu\in\R$, $v_{\pm\mu}$ are weight
vectors and $\mu$ runs through the ``positive'' weights of $V^c$.
If $v$ is killed by $\ker\lambda$, then $c_\mu=0$ unless $\mu$ is a
positive multiple of $\lambda$. Since $V^c$ is minuscule, we deduce that
$\Lg_p\cdot v=0$ implies $v=c_\lambda (v_\lambda+v_{-\lambda})=\pm p$.
\end{proof}

\section{Non-spherical quotients}
\label{S:generalizations}

We start with the characterization of compact Riemannian manifolds $M$ admitting a positive lower bound on the diameter of quotients by isometric actions. The first obvious observation is that if $M$ is non-homogeneous, then such a bound exists, namely $\diam(M/\operatorname{Iso}(M))$. The homogeneous case is Theorem \ref{MT:finitepi1} from the Introduction. To prove it we need the following lemma:
\begin{lemma}
\label{L:semisimple}
Let $G$ be a compact Lie group acting transitively on a compact connected smooth
manifold $M$. Let  $G'=[G^0,G^0]$ be the semi-simple part of $G$. Then $\pi_1(M)$ is finite if and only if $G'$ acts transitively on $M$.
\end{lemma}
\begin{proof}
If $\pi_1(M)$ is finite, then  $G'$ acts transitively  by \cite[Proposition 4.9, page 94]{Onishchik}. Conversely, if $G'$ acts transitively, we choose any $x\in M$, and the long exact sequence of homotopy groups associated to $G'_x\to G'\to M$ implies that $\pi_1(M)$ is finite, because $\pi_1(G')$ is finite and $G'_x$ has finitely many connected components.



\end{proof}

\begin{proof}[Proof of Theorem \ref{MT:finitepi1}]
  Recall $\Isom(M)$ is a compact Lie group and assume first $\pi_1(M)$ is finite. Suppose to the contrary that no such $\epsilon$ exists. Then there exists a sequence of compact non-transitive subgroups $G_i$ of the isometry group of $M$ such that $\lim\diam(M/G_i)=0$. By compactness of the Hausdorff metric, we may assume, after passing to a subsequence, that $G_i$ converges to a compact subset $G_\infty\subset\Isom(M)$. Then $G_\infty$ is a group, and $\diam(M/G_\infty)=0$, that is, $G_\infty$ acts transitively on $M$. By \cite{MontgomeryZippin42}, the groups $G_i$ are eventually conjugate to subgroups of $G_\infty$, so we may assume that $G_i\subset G_\infty$ for all $i$.

Since $M$ has finite fundamental group, we may apply Lemma \ref{L:semisimple} to conclude that the semi-simple part $G'_\infty=[G_\infty^0, G_\infty^0]$ also acts transitively on $M$. The semi-simple parts $G'_i$, being subgroups of $G_i$, also act non-transitively, and thus form a sequence of proper subgroups of $G'_\infty$ that converges to $G'_\infty$. This contradicts \cite[Chapter IV, Proposition 3.7]{tomDieck}, which says that a compact Lie group is a limit of proper subgroups if and only if it is not semi-simple. Therefore an $\epsilon>0$ satisfying the statement of the theorem must exist.

For the converse, assume that $\pi_1(M)$ is infinite. Let $G$ be a finite cover of $\Isom(M)$ of the form $G'\times T^k$, where $G'$ is semi-simple. By Lemma \ref{L:semisimple}, $G'$ does not act transitively on $M$. Therefore neither does any group of the form $G'\times \Gamma$, for $\Gamma$ a finite subgroup of the torus $T^k$. Taking a sequence of finite subgroups $\Gamma_i\subset T^k$ converging to $T^k$, we obtain a sequence of non-transitive subgroups $G_i=G'\times \Gamma_i$ of $G$ such that $\lim_{i\to\infty}\diam(M/G_i)=0$.
\end{proof}

In light of Theorem \ref{MT:finitepi1}, one might suspect that the Main Theorem also generalizes to the class of all compact homogeneous spaces with finite fundamental group (normalized to have a fixed diameter). This turns out to be false, as the next example shows:
\begin{example} \label{E:SO(n)}
Endow $\SO(n)\subset \R^{n^2}$ with the Riemannian metric $g$ induced by the inner product $\left<A,B\right>=\operatorname{tr}(AB^t)/2$ on $\R^{n^2}$. A straight-forward computation shows that the natural quotient map $\SO(n)\to \SO(n)/\SO(n-1)=\Ss^{n-1}$ is a Riemannian submersion, where $\Ss^{n-1}$ is endowed with the standard metric. The diameter of $(\SO(n),g)$ goes to infinity as $n\to\infty$, because it is  bounded from below by the extrinsic diameter as a subset of  $\R^{n^2}$. Indeed, 
\[ \diam(\SO(n),g) \geq d_g(I,-I)\geq d_{\R^{n^2}}(I,-I)=\sqrt{2 n}\]
when $n$ is even, and similarly for $n$ odd. 
\end{example}

\appendix

\section{Proof of Lemma \ref{L:l(S)}} \label{A:l(S)}
Here we prove Lemma \ref{L:l(S)}, which states: There exists a constant $\Co$ such that, for every finite simple group $S$ that is not cyclic or alternating, one has
\[ \frac{\log|S|}{l(S)} \leq C. \]

We use the classification of finite simple groups, see e.g. \cite{GLS, Wilson}. We may discard the sporadic groups, as there are only finitely many of them. The remaining groups are the finite simple groups of Lie type, and come in $16$ families, each parametrized by  a prime power $q$, and possibly a natural number $n$. In \cite{LandazuriSeitz74}, one finds lower bounds for $l(S,q)$ for all $S$ of Lie type, where $l(S,q)$ is defined as the smallest dimension of a projective representation of $G$ over a field of characteristic not dividing $q$. In each family there is a finite number of exceptions to this bound (listed in the third column of the table in \cite[page 419]{LandazuriSeitz74}), which we may and will ignore. Since $l(S)\geq l(S,q)$,  it suffices to show that, in each family, the quotient of $\log |S|$  by the bound provided in \cite{LandazuriSeitz74} is bounded from above. We proceed case by case, following Table 1 from \cite[page 8]{GLS}, and giving first the name as in \cite{GLS}, followed by the name used in \cite{LandazuriSeitz74} (if different). In each case we find an upper bound for the order $|S|$ (whose exact value can be found in \cite[Table 1, page 8]{GLS}), and a lower bound for the lower bound for $l(S,q)$ found in the table in \cite[page 419]{LandazuriSeitz74}.

\begin{enumerate}
\item $A_n(q)=\PSL(n+1,q)$, $n\geq 1$. Then $|S|\leq q^{n^2+n-1}$ and $l(S)\geq (q^{n-1}-1)/2\geq q^{n-1}/4$, so that 
\[\frac{\log|S|}{l(S)}\leq \frac{4(n+1)^2\log(q)}{q^{n-1}}\]
goes to zero.  
\item $^2\!A_n(q)=\PSU(n+1,q)$, $n\geq 2$. Then $|S|\leq 2q^{n^2+n-1}$ and $l(S)\geq (q^n-q)/(q+1)\geq q^{n}/4$, so that 
\[\frac{\log|S|}{l(S)}\leq \frac{4(n+1)^2\log(q)}{q^n}\]
goes to zero.   
\item $B_n(q)=\PSO(2n+1,q)$, $n\geq 3$. Then $|S|\leq q^{2n^2+n}$ and $l(S)\geq q^{2(n-1)}-q^{(n-1)}\geq q^{2(n-1)}/4$, so that 
\[\frac{\log|S|}{l(S)}\leq \frac{4(2n^2+n)\log(q)}{q^{2(n-1)}}\]
goes to zero.  
\item $^2\!B_2(q)=\operatorname{Sz}(q)$. Then $|S|\leq q^5$ and $l(S)\geq \sqrt{q/2}(q-1)\geq q/4$, so that 
\[\frac{\log|S|}{l(S)}\leq \frac{20\log(q)}{q}\]
goes to zero.  
\item $C_n(q)=\PSp(2n,q)$, $n\geq 2$.Then $|S|\leq q^{2n^2+n}$ and $l(S)\geq \min\{q^n-1,q^{n-1}(q^{n-1}-1)(q-1)\}/2\geq q^n/4$, so that 
\[\frac{\log|S|}{l(S)}\leq \frac{4(2n^2+n)\log(q)}{q^n}\]
goes to zero.  
\item $D_n(q)=\PSO^+(2n,q)$, $n\geq 4$. Then $|S|\leq q^{(2n^2-n)}$ and $l(S)\geq q^{2n-3}/2$, so that 
\[\frac{\log|S|}{l(S)}\leq \frac{2(2n)^2\log(q)}{q^{2n-3}}\]
goes to zero.  
\item $ ^2\!D_n(q)=\PSO^-(2n,q)$, $n\geq 4$. Then $|S|\leq 2q^{2n^2-n}$ and $l(S)\geq q^{2n-3}/2$, so that 
\[\frac{\log|S|}{l(S)}\leq \frac{2((2n^2-n)\log(q)+\log(2))}{q^{2n-3}}\]
goes to zero.  
\item $^3\!D_4(q)$. Then $|S|\leq 2 q^{28}$ and $l(S)\geq q^3(q^2-1)\geq q^5/2$, so that 
\[\frac{\log|S|}{l(S)}\leq \frac{2(28\log(q)+\log(2))}{q^5}\]
goes to zero.
\item $G_2(q) $. Then $|S|\leq q^{14}$ and $l(S)\geq q(q^2-1)\geq q^3/2$, so that 
\[\frac{\log|S|}{l(S)}\leq \frac{28\log(q)}{q^3}\]
goes to zero.  
\item $^2\!G_2(q)$. Then $|S|\leq q^7$ and $l(S)\geq q(q-1)\geq q^2/2$, so that 
\[\frac{\log|S|}{l(S)}\leq \frac{14\log(q)}{q^2}\]
goes to zero.  
\item $F_4(q)$. Then $|S|\leq q^{52}$ and $l(S)\geq q^{10}/4$, so that 
\[\frac{\log|S|}{l(S)}\leq \frac{208\log(q)}{q^{10}}\]
goes to zero.  
\item $^2\!F_4(q)$. Then $|S|\leq q^{26}$ and $l(S)\geq q^5/2$, so that 
\[\frac{\log|S|}{l(S)}\leq \frac{52\log(q)}{q^{10}}\]
goes to zero.  
\item $E_6(q)$. Then $|S|\leq q^{78}$ and $l(S)\geq q^{11}/2$, so that 
\[\frac{\log|S|}{l(S)}\leq \frac{156\log(q)}{q^{11}}\]
goes to zero.  
\item $^2\! E_6(q)$. Then $|S|\leq q^{78}$ and $l(S)\geq q^{15}$, so that 
\[\frac{\log|S|}{l(S)}\leq \frac{78\log(q)}{q^{15}}\]
goes to zero.  
\item $E_7(q)$. Then $|S|\leq q^{133}$ and $l(S)\geq q^{17}/2$, so that 
\[\frac{\log|S|}{l(S)}\leq \frac{266\log(q)}{q^{17}}\]
goes to zero.  
\item $E_8(q)$. Then $|S|\leq q^{248}$ and $l(S)\geq q^{29}/2$, so that 
\[\frac{\log|S|}{l(S)}\leq \frac{496\log(q)}{q^{29}}\]
goes to zero.  
\end{enumerate}

\section{Proof of Lemma \ref{L:G^0} -- Analysing the representation of $G^0$}\label{A:G^0}
As in the statement of Theorem \ref{T:irreducible}, let $G^0$ be a product of a torus with finitely many simply-connected compact connected simple Lie groups, $\Gamma$ a finite group acting on $G^0$ by automorphisms, and $G=G^0 \rtimes \Gamma$. Let $\rho:G\to\OO(V)$ be an almost faithful representation whose restriction  $\rho^0:=\rho|_{G^0}$ to $G^0$ is irreducible, and such that $\diam \Ss(V)/G $ is small (but positive). The strategy to prove Lemma \ref{L:G^0} is to show that, up to taking a subgroup of index at most $12$, $G$ is contained in one of the groups listed in Lemma \ref{L:G^0}, so that the statement will follow from Lemma \ref{L:index}. To achieve this, we will use Lemmas \ref{L:normalizer} and \ref{L:super-reducible} to show that the normalizer of $G^0$ in $\OO(V)$ is, up to small index, one of the groups listed in Lemma \ref{L:G^0}.

The proof will consist of a case-by-case analysis, with the division into cases as follows.  First, $G^0$ is either semisimple, of the form  $G^0=G_1\times\cdots\times G_k$ where the $G_i$ are simply-connected compact connected
simple Lie groups and $k\geq1$; or $G^0$ not semisimple, of the form $G^0=\U(1)\times G_1\times\cdots\times G_k$ where the $G_i$ are simply-connected
compact connected simple Lie groups and $k\geq1$. In the latter case, the torus is one-dimensional because the irreducibility of $\rho^0$ implies that the center of $G$ is one-dimensional. As we will see below, $k=1$ will lead to cases (i)--(iii) in the statement of Lemma \ref{L:G^0}, while $k\geq 2$ will lead to cases (iv)--(vii). 

Second, the action of $\Gamma$ on the simple factors may be transitive or not. And third, there is a complex irreducible representation $\pi:G^0\to \U(W)$ such that either one of two cases happen: (i) $\rho^0$ is the real form of $\pi$; or (ii) $\rho^0$ is the realification of $\pi$. Thus there are in principle $8$ cases, but as we will see below, only $4$ may actually occur.

\subsection{$G^0$ semisimple, $\Gamma$-action transitive}

We can write $W=W_1\otimes_{\C}\cdots\otimes_{\C} W_k$. 
where $\pi_i:G_i\to\U({W_i})$ 
is a complex irreducible representation. Since the action of $\Gamma$ on the set of factors of $G^0$ is
transitive, all factors are isomorphic. Fix isomorphisms 
once and for all. Now any two $\pi_i$, $\pi_j$ differ by an  
automorphism of $G_i=G_j$. By composing $\pi_i$ with an automorphism of
$G_i$, we change $\rho^0$ to an orbit-equivalent representation and may assume
all $\pi_i$ equivalent representations. 

{\bf Type (i): $\rho^0$ is a real form of $\pi$.}
If $k=1$, then $G^0$ is simple, and hence its outer automorphism group has order at most $6$. Since $\rho^0$ is of real type,  its centralizer in $\OO(V)$ is $\{\pm 1\}$. Together these imply that the index of $G^0$ in its normalizer in $\OO(V)$ is at most $12$. Now $G^0$ is as in case (i) of Lemma \ref{L:G^0},
and the desired lower bound on $\diam \Ss(V)/G$ is obtained from
Lemma~\ref{L:index}. 

Assume $k\geq 2$. Let $\gamma\in N_{\OO(V)}(G^0)$. Then $\gamma(G_i)=G_{\sigma(i)}$
for all $i$ and a permutation $\sigma\in \Sigma_k$. View 
$\gamma\in N_{\U(W)}(G^0)$ such that $\gamma$ centralizes
the real structure $\epsilon$, which we take 
$\epsilon=\epsilon_1\otimes\cdots\otimes\epsilon_k$, where 
$\epsilon_i$ are ``the same''. Define the complex endomorphism $\gamma_0$
of $W$ by
\[ \gamma_0(w_1\otimes\cdots\otimes w_k)=w_{\sigma(1)}\otimes\cdots\otimes 
w_{\sigma(k)}. \]
Then $\gamma_0$ centralizes $\epsilon$ and normalizes $G^0$
by a simple calculation. 
It follows that $\tilde\gamma:=\gamma\gamma_0^{-1}$ defines a real endomorphism 
of $V$ that normalizes $G_i$ for all $i$. 

Next we distinguish two cases:
\begin{enumerate}[(a)]
\item The $\pi_i$ are of real type.
Here 
$V=V_1\otimes_{\R}\cdots\otimes_{\R} V_k$ where 
$\rho_i:G_i\to \OO(V_i)$ is a real form of $\pi_i$.

We have $\rho|_{G_1}=(\dim_{\R}V_1)^{k-1}\rho_1$.
Lemma \ref{L:normalizer} (real case) now implies that 
$\tilde\gamma\in\OO(V_1)\times\OO(V')$, where 
$V'=V_2\otimes_{\R}\cdots\otimes_{\R} V_k$. 
Proceeding by induction, we see that 
$\tilde\gamma\in\OO(V_1)\times\OO(V_2)\times\cdots\times\OO(V_k)$. Therefore, up to an index $2$ subgroup, we have  $N_{\OO(V)}(G^0)\subset \Sigma_k\ltimes\SO(n)^k$, which appears as case (iv) in Lemma \ref{L:G^0}. Here we may assume $k>2$, because in case $k=2$ the representation is polar
and a lower bound on the diameter of the quotient follows from
Theorem~\ref{T:reflection}.

\item The $\pi_i$ are of quaternionic type and $k$ is even.
Here $\pi_1\otimes\pi_2$ and $\pi_3\otimes\cdots\otimes\pi_k$ are 
of real type. 

We have $\rho|_{G_1\times G_2}=(\dim_{\HH}W_1\otimes_{\HH} W_2)^{\frac k2-1}[\pi_1\otimes\pi_2]_{\R}$ where $[\ \ ]_\R$ denotes a real form.
Lemma \ref{L:normalizer} (real case) now implies that 
$\tilde\gamma\in\OO(W_1\otimes_{\mathbb H} W_2)\times\OO(V')$, where 
$V'$ is a real form of $W_3\otimes_{\C}\cdots\otimes_{\C} W_k$,
with components $\tilde\gamma_{12}$ and $\tilde\gamma_{3\cdots k}$.
Applying Lemma \ref{L:normalizer} (quaternionic case) to $\tilde\gamma_{12}$
and proceeding by induction with $\tilde\gamma_{3\cdots k}$, we see that 
$\tilde\gamma\in\Sp(W_1)\times\cdots\times\Sp(W_k)$. Therefore, $N_{\OO(V)}(G^0)$ is contained in the group listed in part (vi) of Lemma \ref{L:G^0}.
We may assume $k\geq 4$ because in case $k=2$ the
representation is polar.
\end{enumerate}

{\bf Type (ii): $\rho^0$ is the realification of $\pi$.}

If $k=1$, then the identity component of the normalizer of $G^0$
falls into case (ii) or (iii) of Lemma \ref{L:G^0}, according to whether $\rho^0$ is of complex or quaternionic type. Moreover, 
the outer automorphism group of the simple group $G^0$ has order at most
$6$. Therefore the index of $G^0$ in $G$ is bounded by $6$
and the desired lower bound on $\diam \Ss(V)/G$ is obtained from
Lemma~\ref{L:index}. 

Assume $k\geq 2$. Let $\gamma\in N_{\OO(V)}(G^0)$. 
We have $\gamma(G_i)=G_{\sigma(i)}$
for all $i$ and a permutation $\sigma\in \Sigma_k$. 
Define the complex endomorphism $\gamma_0$
of $W$ by
\[ \gamma_0(w_1\otimes\cdots\otimes w_k)=w_{\sigma(1)}\otimes\cdots\otimes 
w_{\sigma(k)}. \]
Then $\gamma_0$ normalizes $G_0$. 
Next we distinguish two cases:
\begin{enumerate}[(a)]
\item The $\pi_i$ are of complex type. \label{cx} 
Here $\pi$ is of complex type.
The element $\gamma$ normalizes the centralizer
of $G^0$, which is to say that $\gamma$ is a complex linear
or conjugate linear endomophism of $W$. By composing with complex conjugation,
we may assume $\gamma$ is complex linear. 
It follows that $\tilde\gamma:=\gamma\gamma_0^{-1}$ is a complex endomorphism 
of $W$ that normalizes $G_i$ for all $i$. 

We have $\pi|_{G_1}=(\dim_{\C} W_1)^{k-1}\pi_1$.
Lemma \ref{L:normalizer} (complex case) says $\tilde\gamma\in\U(W_1)\times\U(W')$,
where $W'=W_2\otimes_{\C}\otimes\cdots\otimes_{\C} W_k$
(recall $\tilde\gamma$ is complex linear). Proceeding by induction, we 
see that $\tilde\gamma\in\U(W_1)\times\cdots\times\U(W_k)$. Therefore, up to a subgroup of index $2$, $N_{\OO(V)}(G^0)$ is contained in the group listed in part (v) of Lemma \ref{L:G^0}.
We may assume $k>2$, since in case $k=2$ the representation is polar.

\item The $\pi_i$ are of quaternionic type and $k$ is odd.

Here $\pi$ is of quaternionic type and $\Sp(1)=Z_{\OO(V)}(G^0)$.
The element $\gamma$ normalizes $\Sp(1)$, and since this group has no 
outer automorphisms, we may assume 
$\gamma$ centralizes it, 
which is to say that $\gamma$ is quaternionic linear. 
Also $\gamma_0$ is quaternionic linear,
so $\tilde\gamma:=\gamma\gamma_0^{-1}$ defines a quaternionic endomorphism 
of $V$ that normalizes $G_i$ for all $i$. 

Write $V=W_1^r\otimes_{\R}V'$, where $W_1^r$ denotes the realification of $W_1$, and $V'$ is a real form
of $W_2\otimes_{\C}\otimes\cdots\otimes_{\C} W_k$. 
Then $\rho|_{G_1}=(\dim_{\R}V')\pi_1^r$ and
Lemma \ref{L:normalizer} (real case) says that $\tilde\gamma\in\OO(W_1^r)\times\OO(V')$,
with components $\tilde\gamma_1$ and $\tilde\gamma'$. 
Now we recall that $\tilde\gamma$ is quaternionic and 
$\tilde\gamma'$ is real to note that indeed 
$\tilde\gamma_1=\tilde\gamma\tilde\gamma'^{-1}\in\Sp(W_1)$, and we 
apply case (i)(b) to $\tilde\gamma'$ to deduce that 
$\tilde\gamma\in\Sp(W_1)\times\cdots\times\Sp(W_k)$. Therefore we
are in case (vii) of Lemma~\ref{L:G^0}, and we note that in case
$k=1$ the representation is polar. 

\end{enumerate}

\subsection{$G^0$ semi-simple, $\Gamma$-action non-transitive}\label{C2}
As above, we can write $W=W_1\otimes_{\C}\cdots\otimes_{\C} W_k$,  
where $\pi_i:G_i\to\U({W_i})$ 
is a complex irreducible representation.
If the action of $\Gamma$ on the set of factors of $G^0$ is
non-transitive, then $k\geq2$ and each orbit produces a connected
normal subgroup of $G$.
We shall shortly see that here we are dealing with Type (i), that is,
$V$ is a real form of $W$.

Write $G^0=G_a\times G_b$, where $G_a$ and $G_b$ are non-trivial
$\Gamma$-invariant subgroups, and write $V=V_a\otimes_\FF V_b$ accordingly. 
Since we are assuming $\diam \Ss(V)/G$ small,
Lemma \ref{L:super-reducible} says that the action
of any normal subgroup of $G$ is either irreducible or super-reducible.
It follows that $G_a$ acts by scalars and $V|_{G_b}$ is irreducible,
up to interchanging $a$ and $b$. Since $G^0$ is connected and semisimple,
this says that $G_a=\Sp(1)$, $V_a=\HH$, $\FF=\HH$ and $V_b$ is of
quaternionic type.
Note that $\Gamma$ must act transitively on the
factors of $G_b$, for otherwise $G_b=\Sp(1)\times G_c$ is a non-trivial
$\Gamma$-invariant decomposition and $\Sp(1)\Sp(1)=\SO(4)$ neither acts
by quaternionic scalars nor is irreducible on $V$.  

Now we can write $G_1=\Sp(1)$ and $W_1=\C^2$, $G_2=\cdots=G_k$, 
$\pi_2=\cdots=\pi_k$ are of quaternionic type with 
respective quaternionic structures $\epsilon_1=\cdots=\epsilon_k$
and $k$ is even.

Let $\gamma\in N_{\OO(V)}(G^0)$. Then $\gamma\in N_{\OO(V')}(G')$
where $G'=G_2\times\cdots\times G_k$ and $V'$ is the realification of 
$W_2\otimes_{\C}\cdots\otimes_{\C}W_k$. By multiplying by an 
element of $\Sp(1)$, we may assume $\gamma$ centralizes $\Sp(1)$. Now we apply 
case (ii)(b) to deduce that $\gamma\in \Sigma_{k-1}\ltimes\Sp(W_2)\times\cdots\times\Sp(W_k)$. Thus $G$ is a subgroup of the group in case (vii) of
Lemma~\ref{L:G^0} (note the different meanings of~$k$ here and there).
Note also that the case $G^0=\Sp(1)\times\Sp(n)$ is transitive
on the unit sphere. 

\subsection{$G^0$ non-semisimple.} As discussed above,
$G^0=\U(1)\times G_1\times\cdots\times G_k$. The representation~$\rho^0$
is necessarily of complex type, so it is the realification of
$\pi:G^0\to\U(W)$. Write 
 $W=\C\otimes_{\C}\otimes W_1\otimes_{\C}\cdots\otimes_{\C} W_k$. 
where $\U(1)$ acts on $\C$ by complex scalar multiplication and 
$\pi_i:G_i\to \U({W_i})$ is a complex irreducible representation.

An argument using Lemma \ref{L:super-reducible}
similar to that in subsection~\ref{C2} shows that
the action of $\Gamma$
on set of the factors of $G^0/\U(1)$ is transitive. 
Let $\gamma\in N_{\OO(V)}(G^0)$. Then $\gamma\in N_{\OO(V')}(G/\U(1))$
and we apply case (ii)(a) to see that 
\[ \gamma\in\Z_2\cdot \Sigma_k\ltimes\U(W_1)\times\cdots\times\U(W_k) \]
where $\Z_2$ acts on $V$ as complex conjugation. Therefore
up to taking a subgroup of index $2$, 
$G$ is a subgroup of the group in case~(v) of Lemma~\ref{L:G^0}.
Note that case $k=1$ is transitive on the unit sphere and
case $k=2$ is polar.

\bibliographystyle{alpha}

\def\cprime{$'$}

\end{document}